\documentclass[12pt,a4paper]{article}
\usepackage{graphicx} 
\usepackage{amsmath,amsthm,amsfonts,amssymb} 
\usepackage{txfonts,eucal,bm} 
\usepackage{color} 

\newtheorem{thm}{Theorem}[section]
\newtheorem{prop}[thm]{Proposition}

\newtheorem{lem}[thm]{Lemma}
\theoremstyle{definition}
\newtheorem{defn}[thm]{Definition}
\theoremstyle{remark}
\newtheorem{rem}[thm]{Remark}
\newtheorem{exa}[thm]{Example}
\numberwithin{equation}{section}

\DeclareMathOperator\adj{adj}%
\DeclareMathOperator\ann{ann}%

\DeclareMathOperator\Char{Char}%
\DeclareMathOperator\GL{GL}%

\DeclareMathOperator\GO{GO}%

\DeclareMathOperator\im{im}%

\DeclareMathOperator\Orth{O}%

\DeclareMathOperator\path{path}
\DeclareMathOperator\spn{span}%
\newcommand{\x}{\times} 
\newcommand{\Trans}{\mathrm T}
\newcommand{\trenn}{{-\hspace{0pt}}}

\newcommand{\li}{\langle} 
\newcommand{\re}{\rangle} 
\newcommand{\lire}{\li\,\cdot\,,\cdot\,\re} 

\newcommand{\Oweak}{\Orth'} 

\newcommand{\hatS}{\hat{\vS}}
\newcommand{\hatR}{\hat{\vR}}
\newcommand{\hatT}{\hat{\vT}}
\newcommand{\hatB}{\hat{B}}
\newcommand{\hatD}{\hat{D}}
\newcommand{\hatG}{\hat{G}}
\newcommand{\hatN}{\hat{N}}
\newcommand{\hatQ}{\hat{Q}}
\newcommand{\hatg}{\hat{g}}

\newcommand{\relB}{\mathrel{\lozenge}} 
\newcommand{\relhatB}{\mathrel{\hat{\lozenge}}} 

\renewcommand{\phi}{\varphi}
\renewcommand{\theta}{\vartheta}
\renewcommand{\kappa}{\varkappa}

\newcommand{\bF}{{\mathbb F}}

\newcommand{\bP}{{\mathbb P}}

\newcommand{\bR}{{\mathbb R}}

\newcommand{\vM}{{\bm M}}

\newcommand{\vR}{{\bm R}}
\newcommand{\vS}{{\bm S}}
\newcommand{\vT}{{\bm T}}

\newcommand{\vV}{{\bm V}}
\newcommand{\vW}{{\bm W}}

\newcommand{\va}{{\bm a}}
\newcommand{\vb}{{\bm b}}
\newcommand{\vc}{{\bm c}}

\newcommand{\ve}{{\bm e}}
\newcommand{\vf}{{\bm f}}

\newcommand{\vm}{{\bm m}}

\newcommand{\vo}{{\bm o}}

\newcommand{\vs}{{\bm s}}

\newcommand{\vu}{{\bm u}}

\newcommand{\vx}{{\bm x}}
\newcommand{\vy}{{\bm y}}

\begin{document}
\sloppy
\author{Hans Havlicek\thanks{https://orcid.org/0000-0001-6847-1544}}
\title{Quadratic forms and their duals}
\date{June 30th, 2025}
\maketitle
\begin{center}
\textit{Dedicated to Friedrich Manhart on the occasion of his 70th birthday}
\end{center}

\begin{abstract}\noindent%
There are many specific results, spread over the literature, regarding the
dualisation of quadrics in projective spaces and quadratic forms on vector
spaces. In the present work we aim at generalising and unifying some of these.
We start with a quadratic form $Q$ that is defined on a subspace $\vS$ of a
finite-dimensional vector space $\vV$ over a field $\bF$. Whenever $Q$
satisfies a certain condition, which comes into effect only when $\bF$ is of
characteristic two, $Q$ gives rise to a \emph{dual quadratic form} $\hatQ$. The
domain of the latter is a particular subspace $\hatS$ of the dual vector space
of $\vV$. The connection between $Q$ and $\hatQ$ is given by a binary relation
between vectors of $\vS$ and linear forms belonging to $\hatS$.
\par\noindent
\textbf{Mathematics Subject Classification (2020):} 15A63  \\
\textbf{Key words:} quadratic form, dual vector space, subspace
\end{abstract}

\section{Introduction}\label{se:intro}
Each finite-dimensional vector space $\vV$ over a field $\bF$ comes along with
its dual space $\vV^*$ and, up to a canonical identification, the dual of
$\vV^*$ is again $\vV$. Beyond that, there are many instances where some
``feature'' defined on $\vV$ determines in a natural way an analogous ``dual
feature'' on $\vV^*$ and, moreover, the initial feature coincides with its
second dual. For example, each basis of $\vV$ gives rise to a unique basis of
$\vV^*$, which is known as its dual basis. In this note we address the problem
whether or not an analogous situation occurs when dealing with quadratic forms.
\par
There are several results that contribute to the above problem. Most of them
are given in terms of quadrics of a finite-dimensional projective space $\bP$
over $\bF$ and its dual space $\bP^*$. Thereby homogenous coordinates are used
to describe not only the points of $\bP$ but also the hyperplanes of $\bP$,
which are identified with the points of $\bP^*$. A quadric locus (resp.\
quadric envelope) is specified as the set of those points of $\bP$ (resp.\
hyperplanes of $\bP$) the homogeneous coordinates of which belong to the zero
set of a chosen quadratic form. A very general construction is given by Hodge
and Pedoe \cite[pp.~213--214]{hodge+p-94b} under the assumption that $\bF$ is
an algebraically closed field of characteristic zero: It is shown that each
quadric locus of a subspace of $\bP$ gives rise to a quadric envelope of a
certain subspace of $\bP^*$. Loosely speaking, the latter quadric envelope
comprises those hyperplanes of $\bP$ that are ``tangent'' to the quadric locus.
The initial quadric locus can be reproduced, \emph{mutatis mutandis}, from its
corresponding quadric envelope. Constructions of the same kind (under varying
assumptions on the initial quadric locus and the ground field) can be found in
many books, for example, \cite[\S~46.3]{burau-61a}, \cite[p.~98
Satz~2]{gier-82a}, \cite[pp.~129--134]{gier-82a},
\cite[Thm.~4.2.3]{odeh+s+g-20a}, \cite[Thm.~4.2.10]{odeh+s+g-20a} and
\cite[pp.~269--270]{semp+k-98a}. We refer also to \cite[p.~155]{stru+s-10a},
where a lattice-theoretical approach is used to describe the dual of a
projective-metric space. Furthermore, the so-called homogeneous model of an
affine-metric space follows these lines. See \cite{bamb+s-25a},
\cite{gunn-17b}, \cite{havl-23a}, \cite{seli-22a} and the many references given
there.
\par
In view of the above, we present a self-contained approach that provides a
one-to-one correspondence between certain quadratic forms defined on subspaces
of $\vV$ and certain quadratic forms defined on subspaces of $\vV^*$. There
will be no restriction whatsoever on the ground field $\bF$. This
correspondence incorporates \emph{all} quadratic forms on subspaces of $\vV$
resp.\ $\vV^*$ provided that $\bF$ is not of characteristic two.
\par
The paper is organised as a kind of guided tour. In Section~\ref{se:pre}, we
collect basic results. Next, in Section~\ref{se:Q-sub}, we exhibit specific
features (in $\vV$ and $\vV^*$) arising from a quadratic form $Q$ that is
defined on a subspace $\vS$ of $\vV$. Section~\ref{se:Q-dach} starts by
requiring a condition to be satisfied by $Q$. This condition holds in any case
when the characteristic of $\bF$ is unequal two. Then the \emph{dual quadratic
form} of $Q$, which is written as $\hatQ$, is defined on a particular subspace
$\hatS$ of $\vV^*$ in a coordinate-free way. As one might expect, $\hatQ$
admits a dual quadratic form, too, which coincides with $Q$. It is worth noting
that the subspace $\hatS$ of $\vV^*$ depends only on the radical $\vR$ of the
polar form of $Q$ and not on the subspace $\vS$. The latter determines the
radical $\hatR$ of the polar form of $\hatQ$. In particular, $\hatR$ is the
zero subspace of $\vV^*$ precisely when $\vS=\vV$. This is why we consider from
the very beginning quadratic forms on subspaces of $\vV$ rather than quadratic
forms defined on all of $\vV$. Still, coordinates must not be avoided.
Therefore, Section~\ref{se:koo} contains the transition from $Q$ to its dual
quadratic form $\hatQ$ in terms of coordinates. Also, some references to
closely related work are given there. The ties between $Q$ and $\hatQ$ are
fostered in Section~\ref{se:sim} by discussing a relationship between the
similarities of $(\vS,Q)$ and $(\hatS,\hatQ)$.

\section{Preliminaries}\label{se:pre}
In this section, we establish notation and collect basic results. For proofs
and further details we refer, among others, to \cite[Ch.~II
\S~2.3--2.7]{bour-89a}, \cite[Kap.~4]{havl-22b}, \cite[Ch.~3]{roman-08a},
\cite[3.7]{shaf+r-13a}, \cite[p.~18]{tayl-92a} and \cite[pp.~50--51]{tayl-92a}.
\par
Hereafter, $\bF$ denotes a (commutative) field. The characteristic of $\bF$,
which is abbreviated as $\Char \bF$, is arbitrary unless explicitly stated
otherwise. A vector space always means a finite-dimensional vector space over
$\bF$.
\par
Given a vector space $\vV$, say, we write $\vV^*$ for the \emph{dual space} of
$\vV$, that is the vector space formed by all linear mappings $\vV\to \bF$. The
elements of $\vV^*$ will be addressed as \emph{linear forms}. The
\emph{canonical pairing} $\lire_{\vV}\colon \vV^* \x \vV\to \bF$ sends any pair
$(\va^*,\vx)\in\vV^*\x\vV$ to the scalar $\li\va^*,\vx\re_{\vV}:=\va^*(\vx)$.
Since the dimension of $\vV$ is finite, we may consider $\vV$ as the dual space
of $\vV^*$ by identifying each vector $\vu\in \vV$ with the mapping
$\li\,\cdot\,,\vu\re_{\vV}\in (\vV^*)^*$ sending $\va^*\mapsto
\li\va^*,\vu\re_{\vV}$ for all $\va^*\in\vV^*$.
\par
If $\vM$ is a subset of $\vV$, then the \emph{annihilator} of $\vM$ with
respect to $\vV$ is given as
\begin{equation}\label{eq:ann}
    \ann_{\vV}(\vM) :=\{\va^*\in \vV^*\mid \li \va^*,\vm\re_{\vV} = 0
    \mbox{~for all~}
    \vm\in\vM\} .
\end{equation}
We note that $\ann_{\vV}(\vM)$ is a subspace of $\vV^*$. Likewise, the
annihilator with respect to $\vV^*$ of any subset of $\vV^*$ is a subspace of
$\vV$. If $\vT$ is a subspace of $\vV$, in symbols $\vT\leq\vV$, then the
dimensions of $\vT$ and $\ann_{\vV}(\vT)$ satisfy $\dim\vT +
\dim\bigl(\ann_{\vV}(\vT)\bigr)=\dim\vV$. By going over to the annihilator of
$\ann_{\vV}(\vT)$ with respect to $\vV^*$, we obtain
$\ann_{\vV^*}\bigl(\ann_{\vV}(\vT)\bigr)=\vT$.
\par
Let $\dim\vV=:n$ and let $\{\ve_1,\ve_2,\ldots,\ve_n\}$ be a basis of $\vV$.
The corresponding \emph{dual basis} of $\vV^*$ is written as
$\{\ve_1^*,\ve_2^*,\ldots,\ve_n^*\}$. Thus, in terms of the Kronecker symbol
$\delta_{ij}$, we have $\li\ve_{i}^*,\ve_{j}\re_{\vV} =\delta_{ij}$ for all
$i,j\in \{1,2,\ldots,n\}$. If $J$ is a subset of $\{1,2,\ldots,n\}$, then
\begin{equation}\label{eq:basis-ann}
    \ann_{\vV}\bigl(\spn \{\ve_j\mid j\in J\} \bigr)
    = \spn\bigl\{\ve_i^*\mid i\in \{1,2,\ldots,n\}\setminus J \bigr\} .
\end{equation}
\par
Suppose that $\lambda\colon\vV\to\vW$ denotes a linear mapping into some vector
space $\vW$. The kernel and the image of $\lambda$ are written as $\ker\lambda$
and $\im\lambda$, respectively. The \emph{transpose} of $\lambda$ is defined to
be $\lambda^\Trans\colon\vW^*\to\vV^*\colon \vc^*\mapsto \vc^*\circ\lambda$.
Thus
\begin{equation*}
    \bigl\li\lambda^\Trans(\vc^*),\vx\bigr\re_{\vV} = \bigl\li \vc^*,\lambda(\vx)\bigr\re_{\vW}
    \mbox{~~for all~~}\vc^*\in\vW^*,\;\vx\in\vV .
\end{equation*}
Consequently, $\lambda^\Trans$ is linear and $(\lambda^\Trans)^\Trans =
\lambda$. Furthermore,
\begin{equation}\label{eq:ker-im}
    \ker\lambda^\Trans = \ann_{\vW}(\im\lambda)
    \mbox{~~and~~}
    \im \lambda^\Trans = \ann_{\vV}(\ker\lambda) .
\end{equation}
\par
The \emph{general linear group} of $\vV$, denoted by $\GL(\vV)$, comprises all
linear bijections from $\vV$ to itself. Let any $\psi\in\GL(\vV)$ be given. We
read off from \eqref{eq:ker-im}, applied to $\psi$, that
$\psi^\Trans\in\GL(\vV^*)$. If $\vT\leq\vV$, then
\begin{equation}\label{eq:psi(T)}
    \psi^\Trans\Bigl(\ann_{\vV}\bigl(\psi(\vT)\bigr)\Bigr)
    =
    \ann_{\vV}(\vT) .
\end{equation}
\par
For the remainder of this note, $\vV$ denotes a vector space that contains a
\emph{distinguished subspace} $\vS\leq\vV$. We also assume that $\vS$ is
equipped with a \emph{quadratic form} $Q\colon \vS\to \bF$. So $Q$ satisfies
two properties: (i) $Q(c\vx)=c^2 Q(\vx)$ for all $c\in \bF$ and all
$\vx\in\vS$. (ii) The mapping
\begin{equation}\label{eq:B}
    B\colon\vS\x\vS\to \bF\colon (\vx,\vy) \mapsto Q(\vx+\vy)-Q(\vx)-Q(\vy)
\end{equation}
is bilinear. We follow the convention to call $(\vS,Q)$ a \emph{metric vector
space}. Only at the beginning of Section~\ref{se:Q-dach}, one additional
assumption concerning $Q$ will be made.
\par
When dealing with $\vS$ or $Q$, we have to draw a clear distinction between
intrinsic and extrinsic notions. An \emph{intrinsic} notion rests upon $\vS$
being a vector space in its own right and disregards $\vV$. An \emph{extrinsic}
notion involves the vector space $\vV$. For example, the vector space $\vS^*$
dual to $\vS$ and the corresponding canonical pairing
$\lire_{\vS}\colon\vS^*\x\vS\to \bF$ are intrinsic notions. Each subset
$\vM\subseteq\vS$ has the (extrinsic) annihilator $\ann_{\vV}(\vM)$ as in
\eqref{eq:ann}. By replacing $\vV$ with $\vS$ and $\lire_{\vV}$ with
$\lire_{\vS}$ in \eqref{eq:ann}, we obtain the (intrinsic) annihilator of $\vM$
with respect to \ $\vS$, which is written as $\ann_{\vS}(\vM)$.

\section{Intrinsic and extrinsic notions coming from $Q$}\label{se:Q-sub}
We proceed by recalling well{\trenn}established intrinsic notions that arise
under the assumptions made at the end of the previous section; see, for
example, \cite[pp.~54--56]{bour-07a}, \cite[pp.~39--40]{elma+k+m-08a},
\cite[Ch.~4]{grove-02a}, \cite[Ch.~12]{grove-02a}, \cite[\S~7.B--C]{schroe-92a}
and \cite[pp.~54--57]{tayl-92a}.
\par
The bilinear form $B$, as defined in \eqref{eq:B}, is addressed as the
\emph{polar form} of $Q$. We note that $B$ is symmetric and satisfies
\begin{equation}\label{eq:2Q}
    B(\vx ,\vx )= 2 Q(\vx ) \mbox{~ for all~} \vx \in \vS .
\end{equation}
The \emph{radical} of $B$ is the subspace
\begin{equation}\label{eq:R}
    \vR :=\bigl\{\vx\in\vS\mid B(\vx,\vy)=0
    \mbox{~for all~} \vy\in \vS\bigr\} \leq \vS .
\end{equation}
We say that $B$ is \emph{non-degenerate} if, and only if, $\vR=\{\vo\}$. The
bilinearity of $B$ implies that the mapping
\begin{equation}\label{eq:D}
    D \colon \vS\to \vS^* \colon \vx \mapsto D(\vx):=B(\vx,\,\cdot\,)
\end{equation}
is well{\trenn}defined and linear. We therefore obtain
\begin{equation}\label{eq:B-D}
    B(\vx,\vy) = \bigl\li D(\vx),\vy \bigr\re_{\vS}
    \mbox{~~for all~~} \vx,\vy\in\vS
\end{equation}
and
\begin{equation}\label{eq:kerD}
    \ker D =
     \bigl\{\vx\in\vS\mid \bigl\li D(\vx),\vy\bigr\re_{\vS} = 0
    \mbox{~for all~}\vy\in\vS\bigr\} = \vR .
\end{equation}
From $B$ being symmetric, we have $\bigl\li D(\vx),\vy\bigr\re_{\vS} =
B(\vx,\vy) = B(\vy,\vx) = \bigl\li D(\vy),\vx\bigr\re_{\vS}$ for all
$\vx,\vy\in \vS$. Therefore and due to the identification of $\vS^{**}$ with
$\vS$, the mapping $D$ coincides with its transpose
$D^\Trans\colon\vS\to\vS^*$. From the second equation in \eqref{eq:ker-im},
applied to $D=D^\Trans$, and from \eqref{eq:kerD}, we get
\begin{equation}\label{eq:imD}
    \im D = \ann_{\vS}(\ker D) =  \ann_{\vS}(\vR) \leq \vS^* .
\end{equation}

\begin{rem}\label{rem:char2}
Let $\Char \bF = 2$. Then \eqref{eq:2Q} implies $B(\vx,\vx)=0$ for all
$\vx\in\vS$, that is, $B$ is an alternating bilinear form. Hence $\dim\vS -
\dim\vR$ is even.
\end{rem}

\begin{rem}\label{rem:modif1}
Let $\Char \bF\neq 2$. Now $Q$ can be recovered from $B$, since \eqref{eq:2Q}
shows $Q(\vx)=\frac{1}{2}B(\vx,\vx)$ for all $\vx\in\vS$. Therefore, under the
premise $\Char \bF\neq 2$, it is customary to address $\frac{1}{2}B$ rather
than $B$ as the polar form of $Q$. However, in order to obtain a unified
approach, we do not follow this convention.
\end{rem}

Next, we introduce several extrinsic notions coming from $(\vS, Q)$ and $\vV$.
In view of results following behind, we thereby adopt the ``reverse'' notation
\begin{equation}\label{eq:SR-dach}
    \hatS := \ann_{\vV}(\vR)\leq\vV^*
    \mbox{~~and~~}\hatR := \ann_{\vV}(\vS)\leq\hatS .
\end{equation}
A crucial concept is the following binary relation on $\hatS\x\vS$, which is
based upon the polar form $B$ of the given quadratic form $Q$.
\begin{defn}\label{defn:relB}
A linear form $\va^*\in\hatS$ is said to be \emph{$B$-linked} to a vector
$\vx\in\vS$, in symbols $\va^*\relB\vx$, precisely when
\begin{equation}\label{eq:relB}
    \li\va^*,\vy\re_{\vV} = B(\vx,\vy)  \mbox{~~for all~~} \vy\in\vS .
\end{equation}
\end{defn}

It will be advantageous to describe the relation $\relB$ in a different way. To
this end, we establish the mapping $N\colon\hatS\to \vS^*\colon \va^*\mapsto
N(\va^*)$ by setting
\begin{equation}\label{eq:N}
    \bigl\li N(\va^*),\vy\bigr\re_{\vS} := \li\va^*,\vy\re_{\vV}
    \mbox{~~for all~~} \vy\in\vS .
\end{equation}
Thus $N(\va^*)$ is just the restriction of the linear form $\va^*$ to $\vS$ or,
informally speaking, $N(\va^*)$ is a ``narrowed'' version of $\va^*$. The
mapping $N$ is linear and
\begin{equation}\label{eq:kerN}
    \ker N = \ann_{\vV}(\vS)=\hatR .
\end{equation}
By the definition of $N$, the image of $N$ is contained in $\ann_{\vS}(\vR)$.
On the other hand, each linear form belonging to $\ann_{\vS}(\vR)$ has a
pre-image under $N$, since it can be extended to at least one linear form
$\vV\to \bF$. Thus $\im N = \ann_{\vS}(\vR)$. Now, from \eqref{eq:imD}, we get
\begin{equation}\label{eq:imN=imD}
    \im N = \ann_{\vS}(\vR) =\im D .
\end{equation}
\par
Next, we express four properties of the relation $\relB$ using the mappings $D$
and $N$, as described in \eqref{eq:D} and \eqref{eq:N}, respectively. The last
two properties can be thought of as a variant form of the \emph{Riesz
representation theorem}; see, for example, \cite[Thm.~11.5]{roman-08a}.
\begin{lem}\label{lem:linear}
The relation $\relB$ from Definition~\ref{defn:relB} satisfies the following
properties:
\begin{enumerate}\itemsep0pt
\item\label{lem:linear.a} Let $\va^*\in\hatS$ and $\vx\in\vS$. Then
    $\va^*\relB\vx$ is equivalent to $N(\va^*)=D(\vx)$.
\item\label{lem:linear.b} Let $\vf_i^*\relB \vs_i$ and $c_i\in \bF$ for
    $i\in\{1,2,\ldots,k\}$. Then $\bigl({\sum_{i=1}^k c_i\vf_i^*}\bigr)
    \relB \bigl({\sum_{i=1}^k c_i\vs_i}\bigr)$.
\item\label{lem:linear.c} Let $\vf^*\in\hatS$. Then there is $\vs\in\vS$
    such that $\{\vx\in \vS\mid \vf^*\relB\vx\}$ equals the coset $\vs+\vR$
    of the subspace $\vR = \ker D \leq \vS$.
\item\label{lem:linear.d} Let $\vs\in\vS$. Then there is $\vf^*\in\hatS$
    such that $\{\va^*\in \hatS\mid \va^*\relB\vs\}$ equals the coset
    $\vf^*+\hatR$ of the subspace $\hatR=\ker N\leq \hatS$.
\end{enumerate}
\end{lem}

\begin{proof}
\eqref{lem:linear.a} We have $\va^*\relB\vx$ precisely when \eqref{eq:relB} is
satisfied. Using \eqref{eq:B-D}, the latter condition can be rewritten as
$\bigl\li N(\va^*),\vy\bigr\re_{\vS} = \bigl\li D(\vx),\vy\bigr\re_{\vS}$ for
all $\vy\in\vS$ or, said differently, $N(\va^*)=D(\vx)$.
\par
\eqref{lem:linear.b} Taking into account the linearity of the mappings $N$ and
$D$, the claim is an immediate consequence of \eqref{lem:linear.a}.
\par
\eqref{lem:linear.c} From \eqref{eq:imN=imD}, there exists $\vs\in\vS$ with
$N(\vf^*)=D(\vs)$. According to \eqref{lem:linear.a}, the given linear form
$\vf^*\in\hatS$ is $B$-linked to a vector $\vx\in\vS$ precisely when
$D(\vx)=D(\vs)$, which is equivalent to $\vx\in\vs+\ker D$. Now the assertion
follows, since \eqref{eq:kerD} shows $\ker D = \vR$.
\par
\eqref{lem:linear.d} We follow the idea of proof from \eqref{lem:linear.c} with
$N$ and $D$ changing their roles. The set of all $\va^*\in\hatS$ that are
$B$-linked to $\vs$ equals the (non-empty) pre-image of $D(\vs)$ under $N$.
From \eqref{eq:kerN}, this pre-image can be written as $\vf^*+\ker N =
\vf^*+\hatR$ for some $\vf^*\in\hatS$.
\end{proof}

\begin{exa}\label{exa:rad-char2}
Let $\dim\vV=3$ and $\dim\vS=2$. We may choose a basis $\{\ve_1,\ve_2,\ve_3\}$
of $\vV$ such that $\vS=\spn\{\ve_1,\ve_2\}$. Also, we assume that
$Q\colon\vS\to \bF$ satisfies
\begin{equation*}
    Q(x_1\ve_1+x_2\ve_2) = x_2^2 \mbox{~~for all~~} x_1,x_2 \in \bF .
\end{equation*}
In order to describe the relation $\relB$ by following
Lemma~\ref{lem:linear}~\eqref{lem:linear.a}, we have to specify the mappings
$N$ and $D$. From \eqref{eq:kerN} and \eqref{eq:basis-ann}, $\ker N = \hatR =
\ann_{\vV}(\vS) = \spn\{\ve_3^*\}$. Next, we determine $\ker D=\vR$, that is,
the radical of $B$. From \eqref{eq:B}, we obtain
\begin{equation*}
    B\left(x_1\ve_1+x_2\ve_2,y_1\ve_1+y_2\ve_2\right) = 2x_2y_2
    \mbox{~~for all~~} x_1,x_2,y_1,y_2 \in \bF .
\end{equation*}
The factor $2\in \bF$ in the preceding formula makes us distinguish two cases.
\par
\emph{Case}~1. Let $\Char \bF=2$. Here $B$ is a zero form. Consequently, $\vR =
\vS$, $D$ is a zero mapping and $\hatS = \hatR$. The last equation implies that
$N$ is a zero mapping, too. Thus $D(\vx)=N(\va)$ or, equivalently,
$\va^*\relB\vx$ holds for all pairs $(\va^*,\vx)\in\hatS\x\vS$.
\par
\emph{Case}~2. Let $\Char \bF\neq 2$. Here $B$ has the radical $\vR =
\spn\{\ve_1\}=\ker D$, whence \eqref{eq:basis-ann} yields
$\hatS=\ann_{\vV}(\vR) = \spn\{\ve_2^*,\ve_3^*\}$. As we noted above, $\ker N =
\spn\{\ve_3^*\}$. Therefore, $\{N(\ve_2^*)\}$ is a basis of $\im N$ and, by
virtue of \eqref{eq:imN=imD}, also a basis of $\im D$. Indeed, $D(\ve_2)=2
N(\ve_2^*)$. To sum up, the following holds for all $a_2,a_3,x_1,x_2\in \bF$:
The linear form $\va^*:=a_2\ve_2^*+a_3\ve_3^*\in\hatS$ is $B$-linked to the
vector $\vx:=x_1\ve_1+x_2\ve_2\in\vS$ if, and only if, $a_2 N(\ve_2^*) =
N(\va^*) = D(\vx) = 2 x_2 N(\ve_2^*)$. In the present setting, this is
equivalent to $a_2=2x_2$.
\begin{figure}[!th]\unitlength1cm
  \centering
  \begin{picture}(7.5,4.0)
  \small
    \put(0,0){\includegraphics[height=8\unitlength]{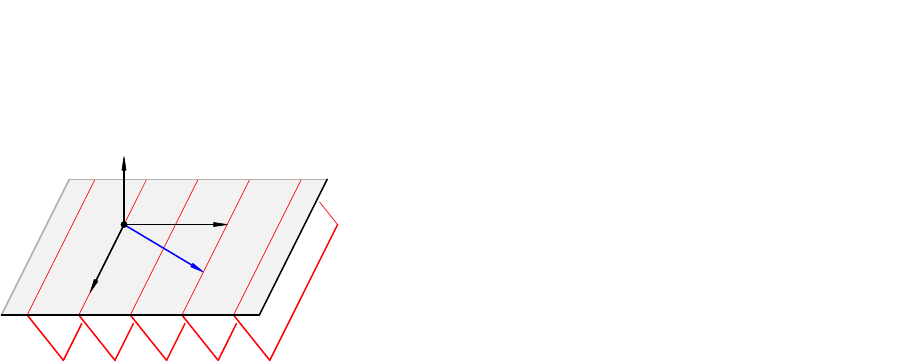}}
    \put(1.65,1.65){$\ve_1$}
    \put(4.7,3.2){$\ve_2$}
    \put(2.85,4.3){$\ve_3$}
    \put(1.5,3.6){$\vS$}
    \color{blue}
    \put(4.1,1.8){$\vs$}
    \color{red}
    \put(0.20,0.45){$\vf^*$}
    \put(1.12,0.45){$-1$}
    \put(2.45,0.45){$0$}
    \put(3.59,0.45){$1$}
    \put(4.73,0.45){$2$}
    \put(5.87,0.45){$3$}
  \end{picture}
   \caption{$B$-linked elements $\vf^*\relB\vs$}
   \label{abb:relB}
\end{figure}
\par
Figure~\ref{abb:relB} illustrates the second case over the real numbers $\bR$.
A linear form $\vf^*:= 2\ve^*_2+f_3\ve^*_3$ with $f_3\in\bR\setminus\{0\}$ is
visualised by some of its level surfaces (parallel planes). Each of these
planes is labelled with a number equal to the constant value assumed there by
$\vf^*$. The linear form $\vf^*$ is $B$-linked to all vectors of the coset
$\ve_2+\spn\{\ve_1\}$, in particular to the highlighted vector
$\vs:=s_1\ve_1+\ve_2$ with $s_1\in\bR\setminus\{0\}$.
\end{exa}

\begin{rem}\label{rem:bij}
The findings in Lemma~\ref{lem:linear} allow for another interpretation: The
linear mapping $N$ gives rise to a linear bijection $N_0$ of the factor space
$\hatS/\hatR = \hatS/\ker N$ onto the subspace $\im N\leq\vS^*$. Explicitly,
$N_0(\va^*+\ker N)=N(\va^*)$ for all $\va^*\in\hatS$. This result is known
under the name \emph{first isomorphism theorem}; see, for example,
\cite[Thm.~3.5]{roman-08a}. Likewise, $D$ yields a linear bijection $D_0$ of
$\vS/\vR = \vS/\ker D$ onto $\im D\leq\vS^*$. From \eqref{eq:imN=imD}, the
product $D_0^{-1}\circ N_0 \colon \hatS/\hatR\to\vS/\vR$ is a linear bijection,
too. The latter mapping takes $\va^*+\hatR\in\hatS/\hatR$ to
$\vx+\vR\in\vS/\vR$ precisely when $\va^*\relB\vx$. Thus, informally speaking,
the relation $\relB$ describes the linear bijection $D_0^{-1}\circ N_0$ by
drawing on representatives of corresponding cosets.
\par
Also, it seems worth noting that $\hatS / \hatR$ is canonically isomorphic to
the dual space of $\vS / \vR$. The canonical pairing then satisfies
\begin{equation*}
    \li\va^*+\hatR,\vx+\vR\re_{\vS/\vR} = \li\va^*,\vx\re_{\vV}
    \mbox{~~for all~~}
    (\va^*,\vx)\in\hatS\x\vS .
\end{equation*}
This follows from a more general result which can
be found in \cite[p.~67]{greub-75a} at the very end of Section 2.23.
\end{rem}

\section{The dual quadratic form of $Q$}\label{se:Q-dach}
While we still adhere to the global settings adopted at the end of
Section~\ref{se:pre}, from now on there will be an additional assumption
regarding the non-zero vectors in the radical $\vR$, namely
\begin{equation}\label{eq:Q-rad=0}
    Q( \vx) = 0 \mbox{~~for all~~}\vx\in\vR\setminus\{\vo\} .
\end{equation}
The reason behind this will be clarified in Remark~\ref{rem:rad=0}. If $\Char
\bF=2$, then there are quadratic forms that fail to satisfy \eqref{eq:Q-rad=0}.
One instance can be read off from Example \ref{exa:rad-char2}, Case~1: The
vector $\ve_2\in\vR$ appearing there satisfies $Q(\ve_2)=1$. If $\Char \bF\neq
2$, then the circumstances are quite different: Indeed, \eqref{eq:R} and
\eqref{eq:2Q} show $Q(\vx)=\frac{1}{2}B(\vx,\vx)=0$ for all $\vx\in\vR$, so
that \eqref{eq:Q-rad=0} holds in any case.
\par
We are now in a good position to present our first main result.
\begin{thm}\label{thm:Q-dach}
Let $\vS\leq \vV$ and let $Q\colon \vS\to \bF$ be a quadratic form subject to
\eqref{eq:Q-rad=0}. Then there exists a unique mapping $\hatQ\colon \hatS\to
\bF$ such that
\begin{equation}\label{eq:Q-dach}
    \va^*\relB\vx   \mbox{~~implies~~}  \hatQ(\va^*) = Q(\vx) .
\end{equation}
This mapping $\hatQ$ is a quadratic form.
\end{thm}

\begin{proof}
Upon choosing $\vf^*\in\hatS$ arbitrarily, there is a vector $\vs\in\vS$ with
$\vf^*\relB\vs$ according to Lemma~\ref{lem:linear}~\eqref{lem:linear.c}. So
there is at most one function $\hatQ$ subject to \eqref{eq:Q-dach}.
Furthermore, $\vf^*\relB\vx$ forces $\vx-\vs\in \vR$. Now, from
\eqref{eq:Q-rad=0} and \eqref{eq:R}, we obtain
\begin{equation}\label{eq:Q-wohldef}
    Q(\vx) = Q\bigl((\vx-\vs)+\vs\bigr)
           = \underbrace{Q(\vx-\vs)}_{=\,0}+Q(\vs)
           + \underbrace{B(\vx-\vs,\vs)}_{=\,0}
           = Q(\vs) .
\end{equation}
Consequently, \eqref{eq:Q-dach} provides a consistent definition of $\hatQ$.
\par
It remains to establish that $\hatQ$ is a quadratic form. First, with
$\vf^*\relB\vs$ as above, let any $c\in \bF$ be given. In order to verify that
$\hatQ(c\vf^*)=c^2\hatQ(\vf^*)$, we make use of
Lemma~\ref{lem:linear}~\eqref{lem:linear.b} to obtain $(c\vf^*) \relB (c\vs)$.
Thus
\begin{equation*}
    \hat Q(c\vf^*) = Q(c\vs) = c^2 Q(\vs) = c^2 \hat Q(\vf^*) ,
\end{equation*}
as required. Next, we consider the mapping
\begin{equation*}
    \hatB\colon \hatS\x\hatS\to \bF \colon
    (\va^*,\vb^*)\mapsto \hatQ(\va^*+\vb^*)-\hatQ(\va^*)-\hatQ(\vb^*) .
\end{equation*}
We claim that
\begin{equation}\label{eq:B-B-dach1}
        (\vf_1^*\relB\vs_1 \mbox{~~and~~}\vf_2^*\relB\vs_2)
        \mbox{~~implies~~}
        \hatB(\vf_1^*,\vf_2^*) = B(\vs_1,\vs_2) .
\end{equation}
The assumptions in \eqref{eq:B-B-dach1} and
Lemma~\ref{lem:linear}~\eqref{lem:linear.b} yield $(\vf_1^*+\vf_2^*) \relB
(\vs_1+\vs_2)$. Hence, from \eqref{eq:Q-dach}, we obtain
\begin{equation*}
    \begin{aligned}
    \hatB(\vf_1^*,\vf_2^*) &= \hatQ(\vf_1^*+\vf_2^*)-\hatQ(\vf_1^*)-\hatQ(\vf_2^*)\\
                           &= Q(\vs_1+\vs_2)-Q(\vs_1)-Q(\vs_2)  = B(\vs_1,\vs_2) .
    \end{aligned}
\end{equation*}
By virtue of \eqref{eq:B-B-dach1} and
Lemma~\ref{lem:linear}~\eqref{lem:linear.b}, it is now straightforward to
derive the bilinearity of $\hatB$ from the bilinearity of $B$.
\end{proof}

We are thus led to the following definition.

\begin{defn}
Let a subspace $\vS\leq\vV$ and a quadratic form $Q\colon\vS\to \bF$ be given
as in Theorem~\ref{thm:Q-dach}. Then the quadratic form $\hatQ\colon\hatS\to
\bF$ appearing there, which is characterised by \eqref{eq:Q-dach}, is called
the \emph{dual quadratic form} of $Q$.
\end{defn}

The dual space of $\vV$ is an essential part of the definition from above. For
example, it has been used in \eqref{eq:SR-dach} to define $\hatS\leq\vV^*$.

\begin{rem}\label{rem:rad=0}
The constraint \eqref{eq:Q-rad=0} has been used in the proof of
Theorem~\ref{thm:Q-dach} in order to establish equation \eqref{eq:Q-wohldef},
which ensures that \eqref{eq:Q-dach} provides a well{\trenn}defined mapping
$\hatQ$. An analogous constraint appears in the closely related context of
induced quadratic forms on factor spaces; see \cite[pp.~56--57]{bour-07a} or
\cite[pp.~56--57]{tayl-92a}. If a quadratic form $Q\colon\vS\to \bF$ does not
meet \eqref{eq:Q-rad=0}, then our definition of the mapping $\hatQ$ fails.
Indeed, suppose that $Q(\vs)\neq 0$ for some $\vs\in\vR\setminus\{\vo\}$.
According to Lemma~\ref{lem:linear}~\eqref{lem:linear.c}, the zero form
$\vo^*\in\hatS$ satisfies not only $\vo^*\relB\vo$ but also $\vo^*\relB\vs$,
whereas $Q(\vo)=0\neq Q(\vs)$.
\end{rem}

We proceed by taking a closer look at the dual quadratic form of $Q$.

\begin{prop}\label{prop:Q-dach}
Let $Q$ be given as in Theorem~\ref{thm:Q-dach}. Then $\hatQ$, the dual
quadratic form of $Q$, and the polar form $\hatB$ of $\hatQ$ satisfy the
following properties:
\begin{enumerate}\itemsep0pt
\item\label{prop:Q-dach.a} If $\vf^*\relB\vs$, then $\hatB(\va^*,\vf^*) =
    \li \va^*,\vs\re_{\vV}$ for all $\va^*\in\hatS$.
\item\label{prop:Q-dach.b} The radical of $\hatB$ equals $\hatR =
    \ann_{\vV}(\vS)$ and $\hatQ(\va^*)=0$ for all
    $\va^*\in\hatR\setminus\{\vo^*\}$.
\end{enumerate}
\end{prop}

\begin{proof}
\eqref{prop:Q-dach.a} From Lemma~\ref{lem:linear}~\eqref{lem:linear.c}, for
each $\va^*\in\hatS$ there is an auxiliary vector $\vu\in\vS$ (depending on
$\va^*$) with $\va^*\relB\vu$. Now \eqref{eq:B-B-dach1} gives
$\hatB(\va^*,\vf^*)=B(\vu,\vs)$, whereas \eqref{eq:relB} implies $B(\vu,\vs) =
\li\va^*,\vs\re_{\vV}$.
\par
\eqref{prop:Q-dach.b} First, we determine the radical of $\hatB$. Let
$\vf^*\in\hatS$. By virtue of Lemma~\ref{lem:linear}~\eqref{lem:linear.c}, we
may pick $\vs\in\vS$ with $\vf^*\relB\vs$. The given linear form $\vf^*$ is in
the radical of $\hatB$ if, and only if, $\hatB(\va^*,\vf^*) = 0$ for all
$\va^*\in\hatS$. According to \eqref{prop:Q-dach.a}, this condition is
equivalent to $\li\va^*,\vs\re_{\vV} = 0$ for all $\va^*\in\hatS$, which in
turn characterises $\vs$ as being in $\vR$. From \eqref{eq:kerD} and
Lemma~\ref{lem:linear}~\eqref{lem:linear.a}, $\vs\in\vR=\ker D$ holds if, and
only if, $D(\vs)=N(\vf^*)$ is a zero linear form. Using \eqref{eq:kerN}, the
latter turns out to be equivalent to $\vf^*\in\ker N = \hatR$.
\par
Next, let $\va^*\in\hatR\setminus\{\vo^*\}$. Then $\li\va^*,\vy\re_{\vV}=0$ for
all $\vy\in\vS$ implies that $\va^*$ is $B$-linked to the zero vector
$\vo\in\vS$. Therefore, $\hatQ(\va^*)=Q(\vo)=0$.
\end{proof}

In view of Proposition~\ref{prop:Q-dach}~\eqref{prop:Q-dach.b}, we may now
apply, \emph{mutatis mutandis}, our previous definitions and results to the
quadratic form $\hatQ$. Thereby $\vV$ and $\vV^*$ swap their roles. The same
applies to $\vS$ and $\hatS$ as well as $\vR$ and $\hatR$, as follows from
$\ann_{\vV^*}(\hatR)=\vS$ and $\ann_{\vV^*}(\hatS)=\vR$, respectively. In
addition, we obtain linear mappings $\hatD\colon \hatS\to(\hatS)^*$ and
$\hatN\colon \vS\to(\hatS)^*$ together with a relation $\relhatB$ on
$\vS\x\hatS$. We note that, in analogy to \eqref{eq:relB}, $\vy\relhatB\vb^*$
is equivalent to
\begin{equation}\label{eq:relBdach}
    \li\va^*,\vy\re_{\vV} = \hatB(\va^*,\vb^*)
    \mbox{~~for all~~} \va^*\in\hatS .
\end{equation}
\par
We now show our second main result.

\begin{thm}\label{thm:Q-dachdach}
Let $Q\colon\vS\to \bF$ be a quadratic form as in Theorem~\ref{thm:Q-dach} and
let $\hatQ\colon\hatS\to \bF$ be its dual quadratic form. Then the following
hold:
\begin{enumerate}\itemsep0pt
\item\label{thm:Q-dachdach.a} The relation $\relB$ on $\hatS\x\vS$, which
    arises from the polar form of $Q$, has the relation $\relhatB$ on
    $\vS\x\hatS$, which arises from the polar form of $\hatQ$, as its
    converse relation.
\item\label{thm:Q-dachdach.b} The dual quadratic form of $\hatQ$ coincides
    with $Q$.
\end{enumerate}
\end{thm}

\begin{proof}
\eqref{thm:Q-dachdach.a} Let any $\vf^*\in\hatS$ be given. Applying
Lemma~\ref{lem:linear}~\eqref{lem:linear.c} and \eqref{lem:linear.d} to the
relations $\relB$ and $\relhatB$, respectively, yields that there are
$\vs,\vs'\in\vS$ with
\begin{equation}\label{eq:s+R}
    \{ \vx\in\vS \mid \vf^*\relB\vx \} = \vs+\vR
    \mbox{~~and~~}
    \{ \vx\in\vS \mid \vx\relhatB\vf^* \} = \vs'+\vR .
\end{equation}
It remains to establish that the above sets coincide. To this end, we apply
Proposition~\ref{prop:Q-dach}~\eqref{prop:Q-dach.a} to $\vf^*\relB\vs$. This
gives $\hatB(\va^*,\vf^*) = \li\va^*,\vs\re_{\vV}$ for all $\va^*\in\hatS$ or,
in view of \eqref{eq:relBdach}, $\vs\relhatB\vf^*$. Now the second equation in
\eqref {eq:s+R} shows $\vs\in\vs'+\vR$, which in turn yields
$\vs+\vR=\vs'+\vR$.
\par
\eqref{thm:Q-dachdach.b} According to part \eqref{thm:Q-dachdach.a}, the
transition from $\hatQ$ to its dual quadratic form will reproduce $Q$.
\end{proof}

\begin{rem}
The bilinear form $B$ (resp.\ $\hatB$) is non-degenerate precisely when
$\hatS=\vV^*$ (resp.\ $\vS=\vV$). If both $B$ and $\hatB$ are non-degenerate,
then (in the terminology of \cite[p.~23 D\'{e}f.~8]{bour-07a}) each of the bilinear
forms $B$ and $\hatB$ is the \emph{inverse} of the other one. Furthermore,
$D^{-1}=\hat{D}$ and, in accordance with Remark~\ref{rem:bij}, the relation
$\relB$ (resp.\ $\relhatB$) is nothing more than the graph of $D^{-1}$ (resp.\
$\hat{D}^{-1}$).
\end{rem}

\begin{rem}\label{rem:modif2}
Let $\Char \bF\neq 2$. As we recalled in Remark~\ref{rem:modif1}, it is common
here to work with $\frac12 B$ rather than $B$. This suggests to introduce a
relation of being \emph{$\bigl(\frac12 B\bigr)$-linked} by replacing $B$ with
$\frac12 B$ in \eqref{eq:relB} and to modify Theorem~\ref{thm:Q-dach}
accordingly. Since $\va^*\relB\vs$ holds precisely when $\va^*$ is
$\bigl(\frac12 B\bigr)$-linked to $2\vs$, and due to $Q(2\vs)=4Q(\vs)$, in this
way $4\hatQ$ (rather than $\hatQ$) will arise from the initial quadratic form
$Q$. Applying this modified construction to $4\hatQ$ reproduces $Q$. We shall
come across $4\hatQ$ again in Remark~\ref{rem:adj}.
\end{rem}

\section{$Q$ and $\hatQ$ in terms of coordinates}\label{se:koo}
In this section the bridge between $(\vS,Q)$ and $(\hatS,\hatQ)$, as considered
in Theorem~\ref{thm:Q-dach}, will be elucidated using coordinates. We thereby
put $\dim\vV=:n$, $\dim\vS=:m$, $\dim\vR=:d$ and we introduce the (possibly
empty) index sets
\begin{equation}\label{eq:I}
\begin{aligned}
    I  & :=\{1,2,\ldots,n\},       &I_{1}&:=\{1,2,\ldots,d\}, \\
    I_{2}& :=\{d+1,d+2,\ldots,m\}, &I_{3}&:=\{m+1,m+2,\ldots,n\} .
\end{aligned}
\end{equation}
There is at least one basis $\{\ve_{i}\mid i\in I\}$ of $\vV$ satisfying
\begin{equation}\label{eq:basis}
    \vR  = \spn\{\ve_{i}\mid i\in I_{1}\} \mbox{~~and~~}
    \vS  = \spn\{\ve_{i}\mid i\in I_{1}\cup I_{2}\} .
\end{equation}
Let $g_{i}:=Q(\ve_{i})$ and $g_{ij}:=B(\ve_{i},\ve_{j})$ for all $i,j\in
I_1\cup I_2$. Then, for all $i,j\in I_1\cup I_2$, formula \eqref{eq:B} gives
$g_{ij}=B(\ve_i,\ve_j)=B(\ve_j,\ve_i)=g_{ji}$ and, furthermore, \eqref{eq:2Q}
gives $2g_i=2Q(\ve_i)=B(\ve_i,\ve_i)=g_{ii}$. (If $\Char\bF =2$, then the last
equation reduces to $0=g_{ii}$; otherwise, it means $g_i=\frac{1}{2}g_{ii}$.
See Remarks~\ref{rem:char2} and \ref{rem:modif1}.) Due to \eqref{eq:Q-rad=0}
and from $\vR$ being the radical of $B$, we have $g_{i}=g_{ij}=g_{ji}=0$
whenever $ i\in I_{1}$. Consequently, for all $x_{1},x_{2},\ldots,x_{m}\in
\bF$, it follows that
\begin{equation}\label{eq:koo-Q}
    Q\left(\sum_{h=1}^{m} x_{h}\ve_{h}\right)
    = \sum_{i=d+1}^{m} g_{i} x_{i}^2
    + \sum_{i=d+1}^{m-1}
      \sum_{j=i+1}^{m} g_{ij} x_{i}x_{j} .
\end{equation}
We now change over to the dual space $\vV^*$. The initially chosen basis of
$\vV$ determines the dual basis $\{\ve_i^*\mid i\in I\}$ of $\vV^*$. We infer
from \eqref{eq:basis} and \eqref{eq:basis-ann} that
\begin{equation*}
    \hatR=\ann_{\vV}(\vS)=\spn\{\ve_{i}^*\mid i\in I_{3}\}
    \mbox{~~and~~}
    \hatS=\ann_{\vV}(\vR)=\spn\{\ve_{i}^*\mid i \in I_{2}\cup I_{3}\} .
\end{equation*}
Let $\hatg_{i}:=\hatQ(\ve_{i}^*)$ and $\hatg_{ij}:=\hatB(\ve_{i}^*,\ve_{j}^*)$
for all $i,j\in I_{2}\cup I_{3}$. From
Proposition~\ref{prop:Q-dach}~\eqref{prop:Q-dach.b}, we get
$\hatg_{i}=\hatg_{ij}=\hatg_{ji}=0$ whenever $i \in I_{3}$. Consequently, for
all $a_{d+1},a_{d+2},\ldots,a_{n}\in \bF$, it follows that
\begin{equation}\label{eq:koo-Qdach}
    \hatQ\left(\sum_{h=d+1}^{n} a_{h}\ve_{h}^*\right)
    = \sum_{i=d+1}^{m} \hatg_{i} a_{i}^2
    + \sum_{i=d+1}^{m-1}
      \sum_{j=i+1}^{m} \hatg_{ij} a_{i}a_{j} .
\end{equation}
\par
Our goal is to express the coefficients $\hatg_{i}$ and $\hatg_{ij}$ appearing
in \eqref{eq:koo-Qdach} using their analogues from \eqref{eq:koo-Q}. To this
end, we introduce the auxiliary subspaces
\begin{equation*}
    \vT := \spn\{\ve_{k}\mid k\in I_{2}\} \leq \vS
    \mbox{~~and~~}
    \hatT := \spn\{\ve_{k}^*\mid k\in I_{2}\} \leq \hatS
\end{equation*}
together with the symmetric matrices
\begin{equation*}
    G_{22}:=(g_{ij})_{i,j\in I_{2}}
    \mbox{~~and~~}
    \hatG_{22}:=(\hatg_{ij})_{i,j\in I_{2}} .
\end{equation*}
From $\vS=\vR\oplus\vT$ (resp.\ $\hatS=\hatR\oplus\hatT$), each element of
$\vT$ (resp.\ $\hatT$) is contained in a single coset of $\vR$ (resp.\
$\hatR$). Thus Lemma~\ref{lem:linear} implies that the restriction of the
relation $\relB$ to $\hatT\x\vT$ is the graph of a linear bijection
$\tau\colon\hatT\to\vT$. (See also Remark~\ref{rem:bij}.) According to
\eqref{eq:D}, for all $k\in I_{2}$, we have $D(\ve_{k}) = {\sum_{l=d+1}^{m}
g_{kl}N(\ve_{l}^*)}$. By virtue of Lemma~\ref{lem:linear}~\eqref{lem:linear.a},
\begin{equation*}
    \left({\sum_{l=d+1}^{m} g_{kl}\ve_{l}^*}\right)
    \relB
    \ve_{k} = \tau\left({\sum_{l=d+1}^{m} g_{kl}\ve_{l}^*}\right)
    \mbox{~~for all~~} k \in I_{2} .
\end{equation*}
So, in terms of the given bases, $G_{22}$ is the matrix of $\tau^{-1}$.
Likewise, the restriction of $\relhatB$ to $\vT\x\hatT$ is the graph of a
linear bijection $\hat{\tau}\colon\vT\to\hatT$, whose inverse is described by
$\hatG_{22}$. From Theorem~\ref{thm:Q-dachdach}~\eqref{thm:Q-dachdach.a},
$\hat{\tau}^{-1} = \tau$, whence
\begin{equation}\label{eq:G22-inv}
    \hatG_{22} = G_{22}^{-1} .
\end{equation}
Furthermore, \eqref{eq:G22-inv} implies $\ve_{i}^* \relB \tau(\ve_{i}^*) =
{\sum_{k=d+1}^{m}\hatg_{ik}\ve_{k}}$ for all $i\in I_{2}$. Thus, from
\eqref{eq:Q-dach} and \eqref{eq:koo-Q}, we finally arrive at
\begin{equation}\label{eq:g-dach}
    \hatg_{i} = Q\left(\sum_{k=d+1}^{m}\hatg_{ik}\ve_{k}\right) \\
                     = \sum_{k=d+1}^{m}g_{k}\hatg_{ik}^2
                     + \sum_{k=d+1}^{m-1}\sum_{l=k+1}^{m}\hatg_{ik}\hatg_{il} g_{kl}
                     \mbox{~~for all~~}i\in I_{2} .
\end{equation}
This rather cumbersome formula incorporates not only the coefficients appearing
in \eqref{eq:koo-Q} but also the entries of the matrix $G_{22}^{-1}$. In the
following part, the preceding results will be simplified under certain
additional prerequisites. However, we have to distinguish two cases.
\par
\emph{Case}~1. Let $\Char \bF = 2$. Now the restriction of $B$ to $\vT\x\vT$ is
a non-degenerate alternating bilinear form, whence $\dim\vT = \dim\vS-\dim\vR =
m-d$ is even; see Remark~\ref{rem:char2}. There is a choice of the basis
vectors $\ve_{k}$ with $k\in I_{2}$ such that the entries of $G_{22}$ along its
minor diagonal are $1$, whereas all remaining entries equal $0$; see
\cite[p.~81 Cor.~3]{bour-07a}, \cite[Thm.~2.10]{grove-02a},
\cite[Satz~9.8.5]{havl-22b}, \cite[Thm.~11.14]{roman-08a} or
\cite[p.~69]{tayl-92a} (sometimes up to a reordering of the basis vectors).
Thus, with $i,j$ ranging in $I_{2}$, we have
\begin{equation}\label{eq:GM-char2}
    g_{ij}=\left\{\begin{array}{ll}
                    1 & \mbox{if~~} j=d+m+1-i ,\\
                    0 & \mbox{otherwise} .
                  \end{array}
            \right.
\end{equation}
Formula \eqref{eq:koo-Q} turns into
\begin{equation*}
    Q\left(\sum_{h=1}^{m} x_{h}\ve_{h}\right)
    = \sum_{i=d+1}^{m} g_{i} x_{i}^2
    + \sum_{i=d+1}^{(d+m)/2}  x_{i}x_{d+m+1-i} .
\end{equation*}
By virtue of $G_{22}=G_{22}^{-1}$, formula \eqref{eq:G22-inv} now reads
$\hatG_{22}=G_{22}$. Consequently, by taking into account \eqref{eq:GM-char2},
formula \eqref{eq:g-dach} reduces to
\begin{equation*}
    \hatg_{i} = \sum_{k=d+1}^{m}g_{k}g_{ik}^2 +
                \sum_{k=d+1}^{m-1}
                \sum_{l=k+1}^{m}\underbrace{g_{ik} g_{il}}_{=\,0} g_{kl}
              = g_{d+m+1-i}
                \mbox{~~for all~~}i\in I_{2} ,
\end{equation*}
whence \eqref{eq:koo-Qdach} can be rewritten as
\begin{equation*}
    \hatQ\left(\sum_{h=d+1}^{n} a_{h}\ve_{h}^*\right)
    = \sum_{i=d+1}^{m} g_{d+m+1-i} a_{i}^2
    + \sum_{i=d+1}^{(d+m)/2} a_{i}a_{d+m+1-i} .
\end{equation*}
\par
\emph{Case}~2. Let $\Char \bF \neq 2$. Now, according to
Remark~\ref{rem:modif1}, we have $Q(\vx)=\frac{1}{2} B(\vx,\vx)$ for all
$\vx\in\vS$ and $\hatQ(\va^*)=\frac{1}{2}\hatB(\va^*,\va^*)$ for all
$\va^*\in\hatS$. In particular, $Q(\ve_{i}) = g_{i} = \frac{1}{2} g_{ii}$ and
$\hatQ(\ve_{i}) = \hatg_{i} = \frac{1}{2} \hatg_{ii}$ for all $i\in I_{2}$.
Moreover, we are able to avoid the usage of \eqref{eq:g-dach} by replacing
\eqref{eq:koo-Q} with
\begin{equation}\label{eq:koo-Q-1/2}
    Q\left(\sum_{h=1}^{m} x_{h}\ve_{h}\right)
    = \frac{1}{2}
      \sum_{i=d+1}^{m}
      \sum_{j=d+1}^{m} g_{ij} x_{i}x_{j}
\end{equation}
and \eqref{eq:koo-Qdach} with
\begin{equation}\label{eq:koo-Qdach-1/2}
    \hatQ\left(\sum_{h=d+1}^{n} a_{h}\ve_{h}^*\right)
    = \frac{1}{2}
      \sum_{i=d+1}^{m}
      \sum_{j=d+1}^{m} \hatg_{ij} a_{i}a_{j} .
\end{equation}
The relationship among the coefficients $g_{ij}$ and $\hatg_{ij}$ appearing in
the above formulas is governed solely by the matrix equation
\eqref{eq:G22-inv}.
\par
Next, we choose the basis vectors $\ve_{k}$ with $k\in I_{2}$ in such a way
that $G_{22}$ is a diagonal matrix; see \cite[pp.~91--92 Cor.~2]{bour-07a},
\cite[Thm.~4.2]{grove-02a}, \cite[Satz~9.8.10]{havl-22b},
\cite[Thm.~11.21]{roman-08a} or \cite[Thm.~6.10]{shaf+r-13a}. By doing so,
\eqref{eq:koo-Q-1/2} simplifies to
\begin{equation*}
    Q\left(\sum_{h=1}^{m} x_{h}\ve_{h}\right)
    = \frac{1}{2} \sum_{i=d+1}^{m} g_{ii} x_{i}^2 .
\end{equation*}
According to \eqref{eq:G22-inv}, the matrix $\hatG_{22}=G_{22}^{-1}$ now is
also diagonal, whence $\hatg_{ii}=g_{ii}^{-1}$ for all $i\in I_{2}$.
Consequently, \eqref{eq:koo-Qdach-1/2} can be rewritten as
\begin{equation*}
    \hatQ\left(\sum_{h=d+1}^{n} a_{h}\ve_{h}^*\right)
    = \frac{1}{2}\,\sum_{i=d+1}^{m} g_{ii}^{-1} a_{i}^2 .
\end{equation*}

\begin{exa}
Let $\vV$ be a five-dimensional vector space over the field $\bR$ of real
numbers. Also, let $\{\ve_1,\ve_2,\ldots,\ve_5\}$ be a basis of $\vV$ and let
$\vS:=\spn\{\ve_1,\ve_2,\ve_3\}$. In view of \eqref{eq:koo-Q-1/2}, we define a
quadratic form $Q\colon\vS\to\bR$ by
\begin{equation*}
    Q\left(\sum_{h=1}^{3}x_h\ve_h\right) :=
    \frac{1}{2}\left(x_2^2 + 2x_2x_3 +2x_3x_2 + 3 x_3^2\right)
    \mbox{~~for all~~}x_1,x_2,x_3\in\bR .
\end{equation*}
The first row of the $3\x 3$ symmetric matrix
\begin{equation*}
    \big(B(\ve_i,\ve_j)\big)_{i,j\in\{1,2,3\}} =
    \left(\begin{array}{ccc}
           0 & 0 & 0\\
           0 & 1 & 2\\
           0 & 2 & 3
    \end{array}\right)
\end{equation*}
shows that $\ve_1$ belongs to the radical $\vR$ of $B$. The above matrix has
rank two, whence $\vR$ has dimension $3-2=1$. Consequently, $\vR=\spn\{\ve_1\}$
so that \eqref{eq:basis} is satisfied. Therefore, using the terminology and
results of the current section, we now have: $n=5$, $m=3$, $d=1$, $I_1=\{1\}$,
$I_2=\{2,3\}$, $I_3=\{4,5\}$, $\hatR=\spn\{\ve_4^*,\ve_5^*\}$ and
$\hatS=\spn\{\ve_2^*,\ve_3^*,\ve_4^*,\ve_5^*\}$. Furthermore,
\begin{equation*}
    G_{22} = (g_{ij})_{i,j\in I_{2}}
           = \left(\begin{array}{cc}
                    1 & 2\\ 2 & 3
                   \end{array}\right)
    \mbox{~~and~~}
    \hatG_{22}  = (\hatg_{ij})_{i,j\in I_{2}}
                = G_{22}^{-1}
                = \left(\begin{array}{rr}
                        -3 & 2\\ 2 & -1
                        \end{array}\right)
\end{equation*}
according to \eqref{eq:G22-inv}. Thus, from \eqref{eq:koo-Qdach-1/2}, the dual
quadratic form $\hatQ\colon\hatS\to\bR$ satisfies
\begin{equation*}
    \hatQ\left(\sum_{h=2}^{5}a_h\ve_h^*\right) =
    \frac{1}{2}\left(-3a_2^2 + 2a_2a_3 + 2a_3a_2 - a_3^2\right)
    \mbox{~~for all~~} a_2,a_3,a_4,a_5\in\bR .
\end{equation*}
\end{exa}

\begin{rem}\label{rem:quadr}
The results of the preceding paragraphs provide (up to scaling factors in
$\bF\setminus\{0\}$) explicit connections to the literature about quadric loci
and their corresponding quadric envelopes, as quoted in Section~\ref{se:intro}.
The same is true for the work on the homogeneous model of an affine metric
space cited directly beneath. Any such model fits into our approach by choosing
$\vS$ as a hyperplane of $\vV$ and $Q$ as a quadratic form with a
non-degenerate polar form $B$. Then $Q$ satisfies \eqref{eq:Q-rad=0} in a
trivial way. The domain of $\hat{Q}$ turns out to be the entire vector space
$\vV^*$ rather than a proper subspace of $\vV^*$. The polar form $\hatB$ has a
one-dimensional radical.
\end{rem}

\begin{rem}\label{rem:adj}
Let $\Char F\neq 2$. Following Remark~\ref{rem:modif2}, let us exhibit the ties
between $Q$ and $4\hat{Q}$ in a different way. Formulas \eqref{eq:koo-Q-1/2}
and \eqref{eq:koo-Qdach-1/2} are actually based upon the matrices
$\frac{1}{2}G_{22}$ and $\frac{1}{2}\hatG_{22}$, respectively. By analogy to
\eqref{eq:koo-Qdach-1/2}, the quadratic form $4\hatQ$ can be expressed using
the matrix $\frac{1}{2}(4\hatG_{22})$, which in turn coincides with the inverse
of the matrix $\frac{1}{2}G_{22}$ describing $Q$. So, in terms of the specific
coordinates underlying our calculations, $Q$ and $4\hatQ$ come, respectively,
from an \emph{invertible symmetric matrix of size $m\x m$} and from its
\emph{inverse matrix}.
\par
When $d=0$ and $m=n$, then the previous observation is well known from the
literature about the equations of quadrics; see the references in
Section~\ref{se:intro}. In addition, it leads to an alternative construction
which relies two facts. First, the inverse of any invertible $n\x n$ matrix $M$
can be written as $(\det M)^{-1}\adj(M)$, where $\adj(M)$ denotes the
\emph{adjoint} of $M$ (the transpose of the cofactor matrix of $M$). Second, a
quadric of a projective space remains unchanged if its defining quadratic form
is multiplied by a non-zero scalar. Therefore, if a quadric locus is defined
(by analogy to the above) in terms of some symmetric (but not necessarily
invertible) $n\x n$ matrix $M$, then the matrix $\adj(M)$ can be used to define
a quadric envelope; see \cite[pp.~58--74]{klein-68a},
\cite[Sect.~9.3--9.4]{rich-11a}, \cite[pp.~102--118]{semp+k-98a} or
\cite[pp.~262--278]{semp+k-98a}. We refer also to \cite[p.~32]{dolg-12a} for
generalisations in the realm of algebraic geometry.
\end{rem}

\section{Extensions of similarities}\label{se:sim}
We start by recalling a few more intrinsic notions coming from the metric
vector space $(\vS,Q)$. See, among others, \cite{elle-77a},
\cite[p.~40]{elma+k+m-08a}, \cite[Ch.~4]{grove-02a}, \cite[Ch.~12]{grove-02a},
\cite[\S~7.G--H]{schroe-92a} and \cite[p.~55]{tayl-92a}. A mapping $\phi$
$\in\GL(\vS)$ is a \emph{similarity} of $(\vS,Q)$ whenever there exists a
constant $c\in \bF\setminus\{0\}$, known as \emph{ratio} of $\phi$, such that
$Q\bigl(\phi(\vx)\bigr)=c Q(\vx)$ for all $\vx\in \vS$. If $Q$ is not the zero
form, then such a ratio $c$ is determined uniquely by $\phi$. Otherwise, in
order to avoid ambiguity, only $c=1$ will be considered as ratio of $\phi$. All
similarities of $(\vS,Q)$ constitute the \emph{general orthogonal group}
$\GO(\vS,Q)$. An \emph{isometry} of $(\vS,Q)$ is, by definition, a similarity
of ratio $1$. All isometries of $(\vS,Q)$ form the \emph{orthogonal group}
$\Orth(\vS,Q)$. The \emph{weak orthogonal group} $\Oweak(\vS,Q)$ consists of
those isometries of $(\vS,Q)$ which fix the radical $\vR$ elementwise.
\par
Next, we adopt the extrinsic point of view. For each $\phi\in\GL(\vS)$ there
exists at least one $\psi\in\GL(\vV)$ with $\psi(\vx)=\phi(\vx)$ for all
$\vx\in\vS$. Any such $\psi$ will be called an \emph{extension} of $\phi$.
\par
The above definitions and results carry over to $(\hatS,\hatQ)$ and
$\GL(\vV^*)$ in an obvious way. The following theorem establishes a
relationship between the groups $\GO(\vS,Q)$ and $\GO(\hatS,\hatQ)$ by
demonstrating that $\psi\in\GL(\vV)$ is an extension of a similarity with ratio
$c$ of $(\vS,Q)$ if, and only if, its transpose $\phi^\Trans\in\GL(\vV^*)$ is
an extension of a similarity with ratio $c$ of $(\hatS,\hatQ)$.

\begin{thm}\label{thm:psi}
Let $Q\colon\vS\to \bF$ be a quadratic form as in Theorem~\ref{thm:Q-dach} and
let $\hatQ\colon\hatS\to \bF$ be its dual quadratic form. Furthermore, let
$\psi\in\GL(\vV)$ and $c\in \bF\setminus\{0\}$. Then the following statements
are equivalent:
\begin{enumerate}\itemsep0pt
  \item\label{thm:psi.a} $\psi(\vS)=\vS$ and $Q\bigl(\psi(\vx)\bigr) = c
      Q(\vx)$ for all $\vx\in\vS$.
  \item\label{thm:psi.b} $\psi^\Trans(\hatS)=\hatS$ and
      $\hatQ\bigl(\psi^\Trans(\va^*)\bigr) = c \hatQ(\va^*)$ for all
      $\va^*\in\hatS$.
\end{enumerate}
\end{thm}

\begin{proof}
\eqref{thm:psi.a}~$\Rightarrow$~\eqref{thm:psi.b} By expressing $B$ in terms of
$Q$ as in \eqref{eq:B}, it follows
\begin{equation}\label{eq:B-phi}
    B\bigl(\psi(\vx),\psi(\vy)\bigr) =
    cB(\vx,\vy)
    \mbox{~~for all~~} \vx,\vy\in\vS .
\end{equation}
Consequently, $\vx\in\vR$ is equivalent to $\psi(\vx)\in\vR$, that is,
$\psi(\vR)=\vR$. From \eqref{eq:psi(T)}, applied to $\vR$, and
$\hatS=\ann_{\vV}(\vR)$, we obtain $ \psi^\Trans(\hatS)= \hatS $.
\par
From Lemma~\ref{lem:linear}~\eqref{lem:linear.c} and $\psi(\vS)=\vS$, for each
$\va^*\in \hatS$ there exists at least one $\vu\in \vS$ (depending on $\va^*$)
such that $\va^*\relB\psi(\vu)$. Then, for all $\vy\in\vS$, it follows that
\begin{equation*}
    B( c\vu, \vy )
    = c B(\vu,\vy )
    = B\bigl(\psi(\vu),\psi(\vy)\bigr)
    = \bigl\li\va^*,\psi(\vy)\bigr\re_{\vV}
    = \bigl\li\psi^\Trans(\va^*),\vy\bigr\re_{\vV} ,
\end{equation*}
where we rest on $B$ being linear in its first argument and on
\eqref{eq:B-phi}. So, according to Definition~\ref{defn:relB},
$\psi^\Trans(\va^*)\relB (c\vu)$. Thus we get
\begin{equation*}
    \hat{Q}\bigl(\psi^\Trans(\va^*)\bigr)
    = Q(c\vu)
    = c^2 Q(\vu)
    = c^2\cdot c^{-1}Q\bigl(\psi(\vu)\bigr) = c\hat{Q}(\va^*) .
\end{equation*}
\par
\eqref{thm:psi.b}~$\Rightarrow$~\eqref{thm:psi.a} In view of
Theorem~\ref{thm:Q-dachdach}~\eqref{thm:Q-dachdach.b} and
$(\psi^\Trans)^\Trans=\psi$, the proof of the converse runs in the same manner
by interchanging the roles of $Q$ and $\hat{Q}$.
\end{proof}

Next, we present an example which illustrates not only Theorem~\ref{thm:psi}
but also yet another property of certain $B$-linked elements.

\begin{exa}
Suppose that $\vf^*\in\hatS$ and $\vs\in\vS$ satisfy $\vf^*\relB \vs$ and
$Q(\vs)\neq 0$. Then $\hatQ(\vf^*)=Q(\vs)\neq 0$, so that $\ker\vf^*$ is a
hyperplane of $\vV$. The linear mapping
\begin{equation*}
    \psi_{\vs,\vf^*}\colon\vV\to\vV\colon
    \vx\mapsto \vx - Q(\vs)^{-1}\li\vf^*,\vx\re_{\vV} \vs
\end{equation*}
fixes all vectors in $\ker\vf^*$ and sends $\vs$ to $-\vs$, since
$\vf^*\relB\vs$ entails $\li\vf^*,\vs\re_{\vV} = B(\vs,\vs)=2Q(\vs)$. Using the
last equation, one readily verifies that $\psi_{\vs,\vf^*}^2$ equals the
identity mapping on $\vV$. Consequently, $\psi_{\vs,\vf^*}\in\GL(\vV)$ is a
\emph{dilatation} or a \emph{transvection}; see \cite[p.~7]{grove-02a} and
\cite[p.~20]{tayl-92a}. The transpose of $\psi_{\vs,\vf^*}$ can be written as
\begin{equation*}
    \psi_{\vs,\vf^*}^\Trans\colon\vV^*\to\vV^*
    \colon \va^*\mapsto \va^*-\hatQ(\vf^*)^{-1} \li\va^*,\vs\re_{\vV}\vf^* .
\end{equation*}
The \emph{reflection} of $(\vS,Q)$ along $\vs$ is that isometry
$\phi_{\vs}\in\Oweak(\vS,Q)$ which is given as
\begin{equation*}
    \phi_{\vs}\colon\vS\to\vS\colon \vx\mapsto \vx - Q(\vs)^{-1}B(\vs,\vx)\vs ;
\end{equation*}
see \cite[p.~40]{elma+k+m-08a}, \cite[p.~40]{grove-02a},
\cite[p.~145]{tayl-92a} or \cite[p.~36]{schroe-92a}. Likewise, the reflection
of $(\hatS,\hatQ)$ along $\vf^*$ reads
\begin{equation*}
    \hat{\phi}_{\vf^*}\colon\hatS\to\hatS\colon
    \va^*\mapsto \va^* - \hatQ(\vf^*)^{-1}\hatB(\va^*,\vf^*)\vf^*
\end{equation*}
and belongs to the weak orthogonal group $\Oweak(\hatS,\hatQ)$. From
$\vf^*\relB\vs$, we have $\li\vf^*,\vx\re_{\vV} = B(\vs,\vx)$ for all
$\vx\in\vS$. Hence $\psi_{\vs,\vf^*}$ is an extension of $\phi_{\vs}$.
Theorem~\ref{thm:psi} ensures that $\psi_{\vs,\vf^*}^\Trans$ extends an
isometry of $(\hatS,\hatQ)$. Looking at
Proposition~\ref{prop:Q-dach}~\eqref{prop:Q-dach.a} reveals that this isometry
equals $\hat{\phi}_{\vf^*}$.
\end{exa}

We wish to express the outcome of Theorem~\ref{thm:psi} in matrix form. In
doing so, $\psi\in\GL(\vV)$ is to satisfy condition \eqref{thm:psi.a} for some
$c\in\bF\setminus\{0\}$. Hence condition \eqref{thm:psi.b} holds for
$\psi^\Trans$. We use again the bases of $\vV$ and $\vV^*$, as introduced at
the beginning of Section~\ref{se:koo}, as well as the index sets $I$, $I_{1}$,
$I_{2}$ and $I_{3}$ from \eqref{eq:I}. Let
\begin{equation*}
    p_{ij} := \bigl\li\ve_i^*,\psi(\ve_j)\bigr\re_{\vV}
            = \bigl\li\psi^\Trans(\ve_i^*),\ve_j\bigr\re_{\vV}
    \mbox{~~for all~~} i,j\in I .
\end{equation*}
The matrix $P:=(p_{ij})_{i,j\in I}$ allows for two interpretations: Reading $P$
column-wise, it describes $\psi$ with respect to basis $\{\ve_{j}\mid j\in I\}$
of $\vV$. Reading $P$ row-wise (or the transpose of $P$ column-wise), it
describes $\psi^\Trans$ with respect to the basis $\{\ve_{i}^*\mid i\in I\}$ of
$\vV^*$. For all $r,s\in\{1,2,3\}$, we denote as $P_{rs}$ that submatrix of $P$
which is formed by all $p_{ij}$ with $i$ and $j$ ranging in $I_{r}$ and
$I_{s}$, respectively. Thereby empty matrices may occur. Since $\psi(\vR)=\vR$
and $\psi(\vS)=\vS$, all entries of the submatrices $P_{21}$, $P_{31}$ and
$P_{32}$ are zero. So the matrix $P$ can be written in block form
\begin{equation*}
    P=\left(
      \begin{array}{ccc}
        P_{11} & P_{12} & \multicolumn{1}{|c}{P_{13}} \\
        \cline{2-3}
        0 & \multicolumn{1}{|c}{P_{22}} & \multicolumn{1}{|c}{P_{23}}\\
        \cline{1-2}
        0 & \multicolumn{1}{|c}{0}   &  P_{33}
      \end{array}
      \right) ,
\end{equation*}
where $0$ serves as an abbreviation for a zero-matrix of an appropriate size.
Moreover, two (overlapping) square submatrices are highlighted. The one in the
upper left corner is based upon the index set $I_{1}\cup I_{2}$ and describes
(column-wise) that similarity of $(\vS,Q)$ which is extended by $\psi$. The one
in the lower right corner is based upon the index set $I_{2}\cup I_{3}$ and
describes (row-wise) that similarity of $(\hatS,\hatQ)$ which is extended by
$\psi^\Trans$. The radical $\vR$ of $B$ (resp.\ $\hatR$ of $\hatB$) is fixed
elementwise under $\psi$ (resp.\ $\psi^\Trans$) precisely when $P_{11}$ (resp.\
$P_{33}$) is an identity matrix. Thus, in general, Theorem~\ref{thm:psi}
provides no relationship between the weak orthogonal groups $\Oweak(\vS,Q)$ and
$\Oweak(\hatS,\hatQ)$.

\begin{rem}
A particular case of Theorem~\ref{thm:psi} is omnipresent in literature about
the homogeneous model of an affine-metric space, as quoted in
Section~\ref{se:intro}. Under the premises sketched in Remark~\ref{rem:quadr},
the theorem leads to an isomorphism of the motion group of the affine-metric
space arising from $(\vS,Q)$ and the weak orthogonal group
$\Oweak(\vV^*,\hatQ)$.
\end{rem}


\noindent
Hans Havlicek\\
Institut f\"{u}r Diskrete Mathematik und Geometrie\\
Technische Universit\"{a}t Wien\\
Wiedner Hauptstra{\ss}e 8--10/104\\
1040 Wien\\
Austria\\
\texttt{havlicek@geometrie.tuwien.ac.at}
\end{document}